\documentclass[10pt]{article}
\usepackage{mathrsfs}
\usepackage{amsfonts}
\usepackage{titlesec}
\usepackage{enumerate}
\usepackage[toc,page,title,titletoc,header]{appendix}
\usepackage{amsthm,amsmath,amssymb,anysize}
\usepackage[numbers,sort&compress]{natbib}
\usepackage{amsmath}
\usepackage{amssymb}
\usepackage{latexsym}
\usepackage{CJK}
\usepackage{longtable}
\usepackage{multirow}
\theoremstyle{definition}
\newtheorem{lemma}{Lemma}
\newtheorem{theorem}[lemma]{Theorem}

\newtheorem{proposition}[lemma]{Proposition}

\newcommand{\periodafter}[1]{#1.}
\titleformat{\subsection}[runin]
{\normalfont\bfseries}{\thesubsection}{0.5em}{\periodafter}
\marginsize{4.2cm}{4.2cm}{3.6cm}{3cm}
\numberwithin{equation}{section}

\renewcommand{\thesubsection}{\arabic{section}.\arabic{subsection}.}

\newcounter{RomanNumber}

\renewcommand{\H}{\mathrm{H}}

\newcommand{\Der}{\mathrm{Der}}
\newcommand{\Ider}{\mathrm{Ider}}
\allowdisplaybreaks[4]

\title{\textsf{The first cohomology of $D(2, 1; \alpha)$ with coefficients in baby Verma modules}}
\author{\textsc{Shuang Lang}$^{a}$, \textsc{Wende Liu}$^{b}$, \textsc{Shujuan Wang}$^{c}$\footnote{Correspondence:  wangsj@shmtu.edu.cn (S.J. Wang)},
\\
\\
\ \ \textit{$^{a}$School of Mathematical Sciences},
  \textit{Dalian University of Technology} \\
  \textit{Dalian 116024. P. R. China}
  \\
\ \ \textit{$^{b}$School of Mathematics and Statistics},
  \textit{Hainan Normal University} \\
  \textit{Haikou 571158. P. R. China}
  \\
\ \ \textit{$^{c}$Department of Mathematics},
  \textit{Shanghai Maritime University} \\
  \textit{Shanghai 201306. P. R. China}
  }
\date{ }

\begin{document}
\maketitle
\begin{quotation}
\small\noindent \textbf{Abstract}:
Over a field of characteristic $p>3,$ the first cohomology group of Lie superalgebra $D(2, 1; \alpha)$ with coefficients in baby Verma modules is determined by calculating the outer superderivations of $D(2, 1; \alpha).$

\vspace{0.2cm} \noindent{\textbf{Keywords}}: $D(2, 1; \alpha)$; baby Verma module; cohomology

\vspace{0.2cm} \noindent{\textbf{Mathematics Subject Classification 2010}}: 17B10, 17B30
\end{quotation}

\setcounter{section}{0}
\section{Introduction}
The cohomology of Lie (super)algebras has many important applications in mathematics and physics.
Kac proposed the idea of determining the first cohomology groups of the simple Lie superalgebras with coefficients in the finite-dimensional irreducible modules over a field of characteristic $0$ (see \cite{Kac}).
Su and Zhang computed the first and second cohomology groups of the classical Lie superalgebras $\mathrm{\mathfrak{sl}}_{m\mid n}$ and $\mathrm{\mathfrak{osp}}_{2\mid 2n}$ with coefficients in the finite-dimensional irreducible modules and Kac modules (see \cite{Su}).
However, most of results on the finite-dimensional Lie (super)algebras over a field of characteristic 0 were not transferable to the modular case (see \cite[Whitehead's Theorem]{NL}, for example).

The cohomology groups are of importance for studying the structure and the classification of modular Lie superalgebras.
Yuan, Liu and Bai determined the second cohomology group with coefficients in the trivial module for the odd Hamiltonian superalgebra and the odd contact superalgebra over a field of characteristic $p>3$ (see \cite{YLB}).
Sun, Liu and Wu obtained the low-dimensional cohomology groups of $\mathrm{sl}_{m\mid n}$ with coefficients in Witt superalgebras over a field of characteristic $p>2$ (see \cite{SLW}).
Wang and Liu determined the first cohomology group of $\mathrm{\mathfrak{sl}}_{2\mid 1}$ with coefficients in the simple modules over a field of characteristic $p>3$ (see \cite{WL}).

In this paper, the field $\mathbb{F}$ is an algebraically closed field of characteristic $p > 3$.
We first characterize the $0$-weight derivation space by the weight space decompositions of $D(2,1;\alpha)$ and baby Verma modules relative to a Cartan subalgebra of $D(2,1;\alpha)$. Then we determine the first cohomology group of Lie superalgebra $D(2, 1; \alpha)$ with coefficients in baby Verma modules over $\mathbb{F}$ by computing the outer superderivations from $D(2,1;\alpha)$ to baby Verma modules in $0$-weight derivation space.

\section{Preliminaries}\label{2}

\subsection{Lie superalgebra cohomology}
We recall some facts as follows:
Let $\mathfrak{g}$ be a Lie superalgebra and $M$ a $\mathfrak{g}$-module. A $\mathbb{Z}_{2}$-homogeneous linear mapping $\varphi: \mathfrak{g}\rightarrow M$ is a derivation of parity $|\varphi|$, if
\begin{equation}\label{222222}
    \varphi([x, y])=(-1)^{|\varphi||x|}x\varphi(y)-(-1)^{|y|(|\varphi|+|x|)}y\varphi(x),
\end{equation}
where $x, y\in \mathfrak{g}$. Let $\mathrm{Der}(\mathfrak{g}, M)$ be the vector space spanned by all the $\mathbb{Z}_{2}$-homogeneous derivations from $\mathfrak{g}$ to $M$ and $\mathrm{Ider}(\mathfrak{g}, M)$ the vector space spanned by all linear mappings $\mathfrak{D}_{m}$
for all $\mathbb{Z}_{2}$-homogeneous elements $m\in M$,
where  $\mathfrak{D}_{m}$ is defined by
 \begin{equation*}\label{111110}
    \mathfrak{D}_{m}(x)=(-1)^{|x||m|}xm
\end{equation*}
for all $x\in \mathfrak{g}$ and $\mathfrak{D}_{m}$ is called the inner derivation with respect to $m$ of parity $|m|$.
Every element of $\Der(\mathfrak{g}, M)$ which does not belong to $\Ider(\mathfrak{g}, M)$ is called an outer superderivation from $\mathfrak{g}$ to $M$.
Suppose that
 \begin{center}
    $\mathfrak{g}=\oplus_{\beta\in\mathfrak{h}^{*}}\mathfrak{g}_{\beta}$ and $M=\oplus_{\beta\in\mathfrak{h}^{*}}M_{\beta}$
 \end{center}
are the root space decompositions of $\mathfrak{g}$ and $M$ with respect to a Cartan subalgebra $\mathfrak{h}$ of $\mathfrak{g}$, respectively.
By \cite[Lemma 2.1]{WL}, we have
\begin{equation*}\label{der}
    \Der(\mathfrak{g}, M)=\Der(\mathfrak{g}, M)_{(0)}+\mathrm{Ider}(\mathfrak{g}, M),
\end{equation*}
where
$$\Der(\mathfrak{g}, M)_{(0)}=\{\varphi\in \Der(\mathfrak{g}, M)\mid \varphi(\mathfrak{g}_{\beta})\subseteq M_{\beta}, \beta\in \mathfrak{h}^{*}\}$$
and $\Der(\mathfrak{g}, M)_{(0)}$ is called the $0$-weight derivation space with respect to $\mathfrak{h}$.
By the definition of the first cohomology group of $\mathfrak{g}$ with coefficients in $M$, we have
 \begin{equation*}\label{h}
 \begin{aligned}
    \H^{1}(\mathfrak{g}, M)&=\Der(\mathfrak{g}, M)/\Ider(\mathfrak{g}, M)\\
    &\cong(\Der(\mathfrak{g}, M)_{(0)}+\Ider(\mathfrak{g}, M))/\Ider(\mathfrak{g}, M)\\
    &\cong\Der(\mathfrak{g}, M)_{(0)}/(\Der(\mathfrak{g}, M)_{(0)}\cap\Ider(\mathfrak{g}, M)).
    \end{aligned}
 \end{equation*}
It allows us to characterize $\H^{1}(\mathfrak{g}, M)$ by computing the outer superderivations from $\mathfrak{g}$ to $M$ in $\Der(\mathfrak{g}, M)_{(0)}$.

\subsection{Lie superalgebra $D(2, 1; \alpha)$}

Let $\mathfrak{g}$ denote the Lie superalgebra $D(2, 1; \alpha)$ with $\alpha \neq 0, -1, \infty$ (see \cite{Sche}).
The Lie superalgebra $\mathfrak{g}$ is 17-dimensional for which
\begin{center}
    $\mathfrak{g}_{\bar{0}}=\mathrm{sl}(2)_{1}\oplus \mathrm{sl}(2)_{2}\oplus \mathrm{sl}(2)_{3}$ and $\mathfrak{g}_{\bar{1}}=V_{1}\boxtimes V_{2}\boxtimes V_{3},$
\end{center}
where $\mathrm{sl}(2)_{i}$ is a copy of $\mathrm{sl}(2)$, $\boxtimes$ is the notation of the outer tensor product and $V_{i}$ is the natural module of $\mathrm{sl}(2)_{i},$ $i=1, 2, 3$. Let $\mathfrak{g}_{\bar{1}}$ be a $\mathfrak{g}_{\bar{0}}$-module by means of
\begin{align*}
  [x_{1}+x_{2}+x_{3}, v_{1}\otimes v_{2}\otimes v_{3}]&:=(x_{1}+x_{2}+x_{3}).v_{1}\otimes v_{2}\otimes v_{3}\\
   &=x_{1}.v_{1}\otimes v_{2}\otimes v_{3}+v_{1}\otimes x_{2}.v_{2}\otimes v_{3}+v_{1}\otimes v_{2}\otimes x_{3}.v_{3},
\end{align*}
where $x_{1}+x_{2}+x_{3}\in \mathfrak{g}_{\bar{0}}$ and $v_{1}\otimes v_{2}\otimes v_{3}\in \mathfrak{g}_{\bar{1}}$.
Let $\{h_{i}, e_{i}, f_{i}\mid i=1,2,3\}$ be a basis of $\mathrm{sl}(2)_{1}\oplus \mathrm{sl}(2)_{2}\oplus \mathrm{sl}(2)_{3}$ satisfying
\begin{equation}\label{gsss}
  [e_{i},f_{i}]=h_{i},\quad [h_{i},e_{i}]=2e_{i},\quad [h_{i},f_{i}]=-2f_{i},
\end{equation}
where $h_{i}, e_{i}, f_{i}\in \mathrm{sl}(2)_{i}$ are viewed as the elements in $\mathrm{sl}(2)_{1}\oplus \mathrm{sl}(2)_{2}\oplus \mathrm{sl}(2)_{3}$ by canonical embedding for $i=1, 2, 3$.
Let $\{\omega_{i}, \omega_{-i}\}$ be a basis of $V_{i},$ $i=1, 2, 3.$ Then $\{\omega_{\pm1}\otimes \omega_{\pm2}\otimes \omega_{\pm3}\}$ is a basis of $V_{1}\boxtimes V_{2}\boxtimes V_{3}.$
For convenience, write
\begin{align*}
 x_{1}&:=\omega_{1}\otimes \omega_{2}\otimes \omega_{3}, & y_{1}&:=\omega_{-1}\otimes \omega_{2}\otimes \omega_{3}, \\
  x_{2}&:=\omega_{1}\otimes \omega_{2}\otimes \omega_{-3}, & y_{2}&:=\omega_{-1}\otimes \omega_{2}\otimes \omega_{-3}, \\
x_{3}&:=\omega_{1}\otimes \omega_{-2}\otimes \omega_{3}, &y_{3}&:=\omega_{-1}\otimes \omega_{-2}\otimes \omega_{3}, \\
 x_{4}&:=\omega_{1}\otimes \omega_{-2}\otimes \omega_{-3}, &y_{4}&:=\omega_{-1}\otimes \omega_{-2}\otimes \omega_{-3}.
\end{align*}
We describe the multiplications of $\mathfrak{g}$ as follows.
The multiplication of $\mathfrak{g}_{\bar{0}}$ follows from (\ref{gsss}).
The multiplication of $\mathfrak{g}_{\bar{0}}$ and $\mathfrak{g}_{\bar{1}}$ is listed in the following:
    \begin{align*}
       [h_{1}, x_{i}]&=x_{i}, &[h_{1}, y_{i}]&=-y_{i},  &[e_{1}, y_{i}]&=x_{i}, &[f_{1}, x_{i}]&=y_{i},\\
       [h_{2}, k_{j}]&=k_{j}, &[h_{2}, k_{l}]&=-k_{l}, &[h_{3}, k_{i}]&=(-1)^{i+1}k_{i}, &[e_{2}, k_{l}]&=k_{l-2},\\
       [f_{2}, k_{j}]&=k_{j+2}, &[e_{3}, k_{s}]&=k_{s-1}, &[f_{3}, k_{t}]&=k_{t+1},
     \end{align*}
where $k_{p}=x_{p}$ or $y_{p}$, $p=i, j, l, s, t$ and $i\in\{1, 2, 3, 4\}$, $j\in\{1,2\}$, $l\in \{3,4\}$, $s\in\{2,4\}$, $t\in\{1,3\}$.
The multiplication of $\mathfrak{g}_{\bar{1}}$ is  listed in the following:
\begin{align*}
[x_{1}, y_{2}]&=-2e_{2},&[x_{1}, y_{3}]&=-2\alpha e_{3},&[x_{1}, y_{4}]&=-(1+\alpha)h_{1}+h_{2}+\alpha h_{3}, \\
 [x_{2}, y_{1}]&=2e_{2}, &[x_{2}, y_{4}]&=2\alpha f_{3},&[x_{2}, y_{3}]&=(1+\alpha)h_{1}-h_{2}+\alpha h_{3},\\
[x_{3}, y_{1}]&=2\alpha e_{3}, &[x_{3}, y_{4}]&=2f_{2}, &[x_{3}, y_{2}]&=(1+\alpha)h_{1}+h_{2}-\alpha h_{3},\\
[x_{4}, y_{2}]&=-2\alpha f_{3}, &[x_{4}, y_{3}]&=-2f_{2},&[x_{4}, y_{1}]&=-(1+\alpha)h_{1}-h_{2}-\alpha h_{3},\\
[y_{2}, y_{3}]&=2(1+\alpha)f_{1}, &[y_{1}, y_{4}]&=-2(1+\alpha)f_{1},&\\
[x_{2}, x_{3}]&=-2(1+\alpha)e_{1}, &[x_{1}, x_{4}]&=2(1+\alpha)e_{1}.
\end{align*}
We define a $p$-mapping $[p]$ on $\mathfrak{g}_{\bar{0}}$ such that $h_{i}^{[p]}=h_{i}, e_{i}^{[p]}=f_{i}^{[p]}=0$, $i=1, 2, 3$. Then $(\mathfrak{g}, [p])$ is a restricted Lie superalgebra.
In addition, $\mathfrak{h}:=\mathbb{F}h_{1}\oplus\mathbb{F}h_{2}\oplus\mathbb{F}h_{3}$ is said to be a Cartan subalgebra of $\mathfrak{g}.$
Let $\epsilon_{i}\in \frak{h}^*$ such that $\epsilon_{i}(h_{j})=\delta_{ij}$, where $\delta_{ij}$ is the Kronecker symbol and $i, j=1, 2, 3$.
We list the weight spaces of $\mathfrak{g}$ with respect to $\mathfrak{h}$ as follows:
\begin{center}
    $\mathfrak{g}_{2\epsilon_{i}}=\mathrm{span}_{\mathbb{F}}\{e_{i}\},$ $\mathfrak{g}_{-2\epsilon_{i}}=\mathrm{span}_{\mathbb{F}}\{f_{i}\},$ $\mathfrak{g}_{\pm\epsilon_{1}\pm\epsilon_{2}\pm\epsilon_{3}}=\mathrm{span}_{\mathbb{F}}\{\omega_{\pm1}\otimes\omega_{\pm2}\otimes\omega_{\pm3}\},$
\end{center}
where $i=1,2,3$.
We may choose a simple root system $\Delta=\{\epsilon_{1}-\epsilon_{2}-\epsilon_{3}, 2\epsilon_{2}, 2\epsilon_{3}\}$ of $\mathfrak{g}.$ Then $\mathfrak{g}$ has a triangular decomposition $$\mathfrak{g}=\mathfrak{n^{+}}\oplus \mathfrak{h} \oplus \mathfrak{n}^{-},$$ where $\mathfrak{n}^{+}=\mathrm{span}_{\mathbb{F}}\{ e_{1}, e_{2}, e_{3}, x_{1}, x_{2}, x_{3}, x_{4}\}$ and $\mathfrak{n}^{-}=\mathrm{span}_{\mathbb{F}}\{ f_{1}, f_{2}, f_{3}, y_{1}, y_{2}, y_{3}, y_{4}\}.$
By \cite[Remark 2.5]{Wang}, we may choose $\chi\in \mathfrak{g}_{\bar{0}}^{\ast}$ with $\chi(\mathfrak{n}_{\bar{0}}^{+})=0$ without loss of generality.
We define a 1-dimensional $\mathfrak{h}$-module $K_{\lambda}$ for each $\lambda\in\mathfrak{h}^{\ast}$, where each $h\in \mathfrak{h}$ acts as multiplication by $\lambda(h).$
Take
\begin{center}
    $\Lambda_{\chi}=\{\lambda\in\mathfrak{h}^{\ast}\mid \lambda(h)^{p}-\lambda(h^{[p]})=\chi(h)^{p}$ for all $h\in \mathfrak{h}\}.$
\end{center}
If $\lambda\in\Lambda_{\chi},$ then $K_{\lambda}$ is a $u_{\chi}(\mathfrak{h})$-module,
 where $u_{\chi}(\mathfrak{h})$ is a $\chi$-reduced universal enveloping algebra of $\mathfrak{h}$.
Recall that the baby Verma module of $\mathfrak{g}$ is $$Z_{\chi}(\lambda):=u_{\chi}(\mathfrak{g})\otimes_{u_{\chi}(\mathfrak{h}\oplus\mathfrak{n}^{+})}K_{\lambda},$$
where $K_{\lambda}$ is regarded as a $u_{\chi}(\mathfrak{h}\oplus\mathfrak{n}^{+})$-module by $\mathfrak{n}^{+}K_{\lambda}=0$  (see \cite{Wang}).
Let $v$ be a basis of $K_{\lambda}.$ Then $Z_{\chi}(\lambda)$ has a homogeneous basis
\begin{equation}\label{11111}
    f_{1}^{i_{1}}f_{2}^{i_{2}}f_{3}^{i_{3}}y_{1}^{j_{1}}y_{2}^{j_{2}}y_{3}^{j_{3}}y_{4}^{j_{4}}\otimes v,
\end{equation}
where $0\leq i_{1}, i_{2}, i_{3}\leq p-1$ and $j_{1}, j_{2}, j_{3}, j_{4}=0, 1.$

\subsection{Target-weight spaces}
The weights
$$\{0, \pm2\epsilon_{1}, \pm2\epsilon_{2}, \pm2\epsilon_{3}, \pm\epsilon_{1}\pm\epsilon_{2}\pm\epsilon_{3}\}$$
are called to be the target-weights.
In order to determine $\mathrm{Der}(\mathfrak{g}, Z_{\chi}(\lambda))_{(0)},$ we will give the weight spaces $Z_{\chi}(\lambda)_{\beta}$ for target-weight $\beta$. For convenience, let
$$J=\{(j_{1}, j_{2}, j_{3}, j_{4})\mid j_{1}, j_{2}, j_{3}, j_{4}=0, 1\}.$$
Let
\begin{equation*}
    \begin{aligned}
    J_{1}&=\{(j_{1}, j_{2}, j_{3}, j_{4})\mid j_{1}+j_{2}+j_{3}+j_{4}\ \mbox{is an even number}\},\\
    J_{2}&=\{(j_{1}, j_{2}, j_{3}, j_{4})\mid j_{1}+j_{2}+j_{3}+j_{4}=2\},\\
    J_{3}&=\{(j_{1}, j_{2}, j_{3}, j_{4})\mid j_{1}+j_{2}+j_{3}+j_{4}\ \mbox{is an odd number}\},\\
    J_{4}&=\{(j_{1}, j_{2}, j_{3}, j_{4})\mid j_{1}+j_{2}+j_{3}+j_{4}=3\}.
    \end{aligned}
\end{equation*}
Note that $J_{1}, J_{2}, J_{3}, J_{4}$ are subsets of $J$.

\begin{lemma}\label{222}
Suppose that $\lambda=\lambda_{1}\epsilon_{1}+\lambda_{2}\epsilon_{2}+\lambda_{3}\epsilon_{3}, \lambda_{1}, \lambda_{2}, \lambda_{3}\in \mathbb{F}$ and $\beta\in\{0, \pm2\epsilon_{i}, \pm\epsilon_{1}\pm\epsilon_{2}\pm\epsilon_{3}\mid i=1, 2, 3\}.$ Suppose that $\beta=\beta_{1}\epsilon_{1}+\beta_{2}\epsilon_{2}+\beta_{3}\epsilon_{3},$ where $\beta_{i}\in \mathbb{F},$ $i=1,2,3$.
For any element $(j_{1},j_{2},j_{3},j_{4})\in J$, denoted by $\theta$. Let $w_{\beta}^{\theta}$ denote $f_{1}^{\beta;\theta;1}f_{2}^{\beta;\theta;2}f_{3}^{\beta;\theta;3}y_{1}^{j_{1}}y_{2}^{j_{2}}y_{3}^{j_{3}}y_{4}^{j_{4}}\otimes v$, where $\beta;\theta;i$ satisfies the following equalities
\begin{alignat*}{2}
     &\beta;\theta;1=\frac{\lambda_{1}-\beta_{1}-j_{1}-j_{2}-j_{3}-j_{4}}{2},\\
     &\beta;\theta;2=\frac{\lambda_{2}-\beta_{2}+j_{1}+j_{2}-j_{3}-j_{4}}{2},\\
     &\beta;\theta;3=\frac{\lambda_{3}-\beta_{3}+j_{1}-j_{2}+j_{3}-j_{4}}{2},
   \end{alignat*}
and $\beta;\theta;k$ are the smallest nonnegative integers in the residue class containing $\beta;\theta;k$ modulo $p$ for $k=1,2,3$.
Then the target-weight spaces $$Z_{\chi}(\lambda)_{\beta}= \mathrm{span}_{\mathbb{F}}\{w_{\beta}^{\theta}\mid \theta\in J\}.$$

\end{lemma}

\begin{proof}
We compute the weight of $f_{1}^{a_{1}}f_{2}^{a_{2}}f_{3}^{a_{3}}y_{1}^{j_{1}}y_{2}^{j_{2}}y_{3}^{j_{3}}y_{4}^{j_{4}}\otimes v$ is
\begin{equation*}
    \begin{aligned}
    &(-2a_{1}-j_{1}-j_{2}-j_{3}-j_{4}+\lambda_{1})\epsilon_{1}+(-2a_{2}+j_{1}+j_{2}-j_{3}-j_{4}+\lambda_{2})\epsilon_{2}\\
&+(-2a_{3}+j_{1}-j_{2}+j_{3}-j_{4}+\lambda_{3})\epsilon_{3}
    \end{aligned}
\end{equation*}
for $a\in \mathbb{F}$ and $j_{1},j_{2},j_{3},j_{4}=0,1$.
Suppose that $w_{\beta}\in Z_{\chi}(\lambda)_{\beta}$ is a weight vector associated with the weight $\beta$.
Then \begin{alignat*}{2}
     &a_{1}=\frac{\lambda_{1}-\beta_{1}-j_{1}-j_{2}-j_{3}-j_{4}}{2},\\
     &a_{2}=\frac{\lambda_{2}-\beta_{2}+j_{1}+j_{2}-j_{3}-j_{4}}{2},\\
     &a_{3}=\frac{\lambda_{3}-\beta_{3}+j_{1}-j_{2}+j_{3}-j_{4}}{2},
   \end{alignat*}

If $0\leq a_{1}, a_{2}, a_{3}\leq p-1$, then $f_{1}^{a_{1}}f_{2}^{a_{2}}f_{3}^{a_{3}}y_{1}^{j_{1}}y_{2}^{j_{2}}y_{3}^{j_{3}}y_{4}^{j_{4}}\otimes v$ is a basis element of $Z_{\chi}(\lambda)_{\beta}$ by (\ref{11111}).

If there exists $k\in\{1, 2, 3\}$ such that $a_{k}>p-1$ or $a_{k}<0$, then replacing $a_{k}-mp$ by $\beta;\theta;k$ and relabelling yields $0\leq \beta;\theta;k\leq p-1$ for $m\in \mathbb{Z}$.
Since $$f_{k}^{a_{k}}=f_{k}^{\beta;\theta;k}f_{k}^{pm}=\chi(f_{k})^{pm}f_{k}^{\beta;\theta;k}.$$
Hence $w_{\beta}$ always consists of $w_{\beta}^{\theta}$ and $\beta;\theta;k$ denote the smallest nonnegative integer in the residue class containing $a_{k}$ modulo $p$ for $k=1, 2, 3$.
\end{proof}

\section{The first cohomology group}
In this section, let $\mathfrak{g}=D(2,1;\alpha)$ and $Z_{\chi}(\lambda)$ be a baby Verma module of $\mathfrak{g}$.
Using the notation introduced in section \ref{2} we obtain the following results.

\begin{lemma}
Suppose that $\varphi\in \Der(\mathfrak{g}, Z_{\chi}(\lambda))_{(0)}\cap\Ider(\mathfrak{g}, Z_{\chi}(\lambda)).$ Then $$\varphi=\sum_{\theta\in J}b^{\theta}\mathfrak{D}_{\omega_{0}^{\theta}},$$
where $b^{\theta}\in \mathbb{F}$.
\end{lemma}
\begin{proof}
Since $\Der(\mathfrak{g}, Z_{\chi}(\lambda))$ is an $\mathfrak{h}^{*}$-graded Lie superalgebra and $\Ider(\mathfrak{g}, Z_{\chi}(\lambda))$ is an $\mathfrak{h}^{*}$-graded subalgebra of $\Der(\mathfrak{g}, Z_{\chi}(\lambda))$. Thus
$$\Der(\mathfrak{g}, Z_{\chi}(\lambda))_{(0)}\cap\Ider(\mathfrak{g}, Z_{\chi}(\lambda))=\Ider(\mathfrak{g}, Z_{\chi}(\lambda))_{(0)},$$
where $\Ider(\mathfrak{g}, Z_{\chi}(\lambda))_{(0)}=\{\sigma\in \Ider(\mathfrak{g}, Z_{\chi}(\lambda))\mid \sigma(\mathfrak{g}_{\beta})\subseteq Z_{\chi}(\lambda)_{\beta}, \beta\in \mathfrak{h}^{*}\}.$ Therefore, we obtain
$$\Ider(\mathfrak{g}, Z_{\chi}(\lambda))_{(0)}=\{\sum_{\theta\in J}b^{\theta}\mathfrak{D}_{\omega_{0}^{\theta}}\mid b^{\theta}\in \mathbb{F}\},$$ since $\mathfrak{D}_{m}$ is a homogeneous derivation of parity $|m|$ for $m\in Z_{\chi}(\lambda)$.
\end{proof}
We give a description of $\mathrm{Der}(\mathfrak{g}, Z_{\chi}(\lambda))_{(0)}.$
Hereafter, write
\begin{equation*}
    \chi(x)^{p}_{s,t}=\left\{
               \begin{array}{ll}
                 1 & \hbox{$s\neq t$},\\
                 \chi(x)^{p} & \hbox{$s=t$},\\
        \end{array}
       \right.
   \end{equation*}
for $x\in \mathfrak{g}$ and $s,t\in \mathbb{F}$. For $\lambda=\lambda_{1}\epsilon_{1}+\lambda_{2}\epsilon_{2}+\lambda_{3}\epsilon_{3}\in \mathfrak{h}^{*}$, we write in short $\lambda=(\lambda_{1}, \lambda_{2}, \lambda_{3})$.

\begin{lemma}\label{3.4}
Suppose that $\varphi\in \mathrm{Der}(\mathfrak{g}, Z_{\chi}(\lambda))_{(0)}.$ Then we have
\begin{equation*}
    \varphi(h_{i})=\left\{
               \begin{array}{ll}
                 a_{i}w_{0}^{(1,1,1,1)} & \hbox{if $\lambda=(2p+2, 2p-2, 2p-2)$ and $\chi(f_{i})^{p}=0$},\\
                 0 & \hbox{otherwise},\\
        \end{array}
       \right.
   \end{equation*}
where $a_{i}\in \mathbb{F}$ and $i=1, 2, 3.$
\end{lemma}

\begin{proof}
Note that $\varphi(h_{i})\in Z_{\chi}(\lambda)_{0}.$ By Lemma \ref{222}, suppose that $\varphi(h_{i})=\sum_{\theta}a_{i}^{\theta}w_{0}^{\theta}$ and $\varphi(f_{1})=\sum_{\theta}a_{-2\epsilon_{1}}^{\theta}w_{-2\epsilon_{1}}^{\theta},$ where $a_{i}^{\theta}, a_{-2\epsilon_{i}}^{\theta}\in\mathbb{F}$ and $\theta\in J$.
By the definition of $\beta;\theta;k$ in Lemma \ref{222}, $k=1, 2, 3,$ we have
\begin{equation*}
 0;\theta;1=\frac{\lambda_{1}+j_{1}-j_{2}+j_{3}-j_{4}}{2}.
\end{equation*}

(1) Suppose that $\lambda\in \mathfrak{h}^{*}$ and $\lambda\neq(2p+2, 2p-2, 2p-2)$.
By substituting $x=h_{i}$ and $y=f_{1}$ in (\ref{222222}), we have
\begin{equation}\label{33333}
\begin{aligned}
    \delta_{i,1}\varphi(-2f_{1})&=\varphi([h_{i}, f_{1}])=h_{i}\varphi(f_{1})-f_{1}\varphi(h_{i})\\
    &=h_{i}\sum_{\theta}a_{-2\epsilon_{1}}^{\theta}w_{-2\epsilon_{1}}^{\theta}-f_{1}\sum_{\theta}a_{i}^{\theta}w_{0}^{\theta}\\
    &=-\delta_{i,1}2\varphi(f_{1})-\sum_{\theta}\chi(f_{1})^{p}_{(0;\theta;1), p-1}a_{i}^{\theta}f_{1}w_{0}^{\theta}.
    \end{aligned}
   \end{equation}
Since $\delta_{i,1}\varphi(-2f_{1})=-\delta_{i,1}2\varphi(f_{1})$, then $\sum_{\theta}\chi(f_{1})^{p}_{(0;\theta;1), p-1}a_{i}^{\theta}f_{1}w_{0}^{\theta}=0$.

\textbf{Case 1}. If $0;\theta;1\neq p-1$ for all $\theta\in J$, then
$$\sum_{\theta}\chi(f_{1})^{p}_{(0;\theta;1), p-1}a_{i}^{\theta}f_{1}w_{0}^{\theta}=\sum_{\theta}a_{i}^{\theta}f_{1}w_{0}^{\theta}.$$
Since $f_{1}w_{0}^{\theta}=w_{-2\epsilon_{1}}^{\theta}$ and $w_{-2\epsilon_{1}}^{\theta}$ are basis elements of $Z_{\chi}(\lambda),$ we have $a_{i}^{\theta}=0.$ Hence $\varphi(h_{i})=0,$ $i=1, 2, 3$.

\textbf{Case 2}. If there exists $J'\subseteq J$ such that $0;\theta';1=p-1,$ $\theta'\in J'$, then $a_{i}^{\theta}=0$ and $\sum_{\theta'\in J'}\chi(f_{1})^{p}a_{i}^{\theta'}w_{-2\epsilon_{1}}^{\theta'}=0$ by (\ref{33333}), where $\theta\in J\backslash J'$.
If $\chi(f_{1})^{p}\neq0$, then $a_{i}^{\theta'}=0$, that is, $\varphi(h_{i})=0,$ $i=1, 2, 3$.
If $\chi(f_{1})^{p}=0$, then substituting $x=h_{i}$ and $y=y_{j}$ or $e_{1}$ in (\ref{222222}), where $j=1, 2, 3, 4$.
A direct computation shows that $\varphi(h_{i})=0,$ $i=1, 2, 3$.

(2) Suppose that $\lambda=(2p+2, 2p-2, 2p-2)$. Substituting $h_{i}$ for $x$ and all basis elements for $y$ in (\ref{222222}). If there exists $k\in \{1,2,3\}$ such that $\chi(f_{k})^{p}\neq0$, then we obtain that $a_{i}^{\theta}=0$, $\theta\in J$, that is, $\varphi(h_{i})=0,$ $i=1, 2, 3$.
If $\chi(f_{i})^{p}=0,$ $i=1,2,3$, then we have $a_{i}^{\theta}=0$, where $\theta\in J\backslash\{(1,1,1,1)\}$. Hence $\varphi(h_{i})=a_{i}^{(1, 1, 1, 1)}w_{0}^{(1, 1, 1, 1)}$, where $a_{i}^{(1, 1, 1, 1)}\in \mathbb{F}$ and $i=1, 2, 3$. Let $a_{i}$ denote $a_{i}^{(1,1,1,1)}$ and we obtain the desired result.
\end{proof}
Hereafter, we denote $a_{\beta}^{(1,1,1,1)}$ by $a_{\beta}$ for all $\beta\in \mathfrak{h}^{*}$.
\begin{lemma}\label{33336}
Suppose that $\varphi(f_{i})=\sum_{\theta}a_{-2\epsilon_{i}}^{\theta}w_{-2\epsilon_{i}}^{\theta},$ where $a_{-2\epsilon_{i}}^{\theta}\in \mathbb{F},$ $i=1, 2, 3$ and $\theta\in J.$ Then
$$\chi(f_{l})^{p}_{(-2\epsilon_{k};\theta;l),p-1}a_{-2\epsilon_{k}}^{\theta}=\chi(f_{k})^{p}_{(-2\epsilon_{l};\theta;k),p-1}a_{-2\epsilon_{l}}^{\theta},$$
where $k,l\in\{1, 2, 3\}$ and $k\neq l$.

\end{lemma}
\begin{proof}
Suppose that $k,l\in\{1, 2, 3\}$ and $k\neq l$. By substituting $x=f_{k}$ and $y=f_{l}$ in (\ref{222222}), we have
\begin{equation}\label{3332}
\begin{aligned}
0=\varphi([f_{k}, f_{l}])&=f_{k}\varphi(f_{l})-f_{l}\varphi(f_{k})\\
&=\sum_{\theta}(\chi(f_{k})^{p}_{(-2\epsilon_{l};\theta;k),p-1}a_{-2\epsilon_{l}}^{\theta}
-\chi(f_{l})^{p}_{(-2\epsilon_{k};\theta;l),p-1}a_{-2\epsilon_{k}}^{\theta})w_{-2\epsilon_{k}-2\epsilon_{l}}^{\theta}.
\end{aligned}
\end{equation}
Since $w_{-2\epsilon_{k}-2\epsilon_{l}}^{\theta}$ is a basis element of $Z_{\chi}(\lambda)$, then (\ref{3332}) yields
$$\chi(f_{l})^{p}_{(-2\epsilon_{k};\theta;l),p-1}a_{-2\epsilon_{k}}^{\theta}
=\chi(f_{k})^{p}_{(-2\epsilon_{l};\theta;k),p-1}a_{-2\epsilon_{l}}^{\theta}.$$
\end{proof}

\begin{proposition}\label{33335}
Suppose that $\varphi\in \Der(\mathfrak{g}, Z_{\chi}(\lambda))_{(0)}.$ If $\chi(f_{1})^{p}\neq0$, then $$\varphi\in\Ider(\mathfrak{g}, Z_{\chi}(\lambda))_{(0)}.$$
\end{proposition}
\begin{proof}
Since $\mathrm{Der}(\mathfrak{g}, Z_{\chi}(\lambda))_{(0)}$ is a $\mathbb{Z}_{2}$-graded vector space, we have $\varphi=\varphi_{\bar{0}}+\varphi_{\bar{1}},$ where $\varphi_{\bar{i}}\in \Der(\mathfrak{g}, Z_{\chi}(\lambda))_{(0), \bar{i}}$ and
$$\Der(\mathfrak{g}, Z_{\chi}(\lambda))_{(0), \bar{i}}:=\Der(\mathfrak{g}, Z_{\chi}(\lambda))_{(0)}\cap\Der(\mathfrak{g}, Z_{\chi}(\lambda))_{\bar{i}}$$
for $\bar{i}\in \mathbb{Z}_{2}$.
Let $x$ be a weight vector and $\beta_{x}$ the weight associated with $x$.
By Lemma \ref{222}, we assume that
\begin{equation}\label{33337}
\begin{aligned}
    \varphi_{\bar{0}}(x_{\bar{0}})&=\sum_{\theta\in J_{1}}a_{\beta_{x}}^{\theta}w_{\beta_{x}}^{\theta}, &\varphi_{\bar{1}}(x_{\bar{0}})=\sum_{\theta\in J_{3}}a_{\beta_{x}}^{\theta}w_{\beta_{x}}^{\theta},\\
\varphi_{\bar{0}}(x_{\bar{1}})&=\sum_{\theta\in J_{3}}a_{\beta_{x}}^{\theta}w_{\beta_{x}}^{\theta}, &\varphi_{\bar{1}}(x_{\bar{1}})=\sum_{\theta\in J_{1}}a_{\beta_{x}}^{\theta}w_{\beta_{x}}^{\theta}.
\end{aligned}
\end{equation}

\textbf{Case 1}. Suppose that $\beta;\theta;1\neq p-1$ for all $\beta\in\{\pm2\epsilon_{i}, \pm\epsilon_{1}\pm\epsilon_{2}\pm\epsilon_{3}\}$.
Replacing $x$ by $f_{1}$ and $y$ by other basis elements in (\ref{222222}),
we obtain that each coefficient $a_{\beta_{x}}^{\theta}$ is a linear combination of the elements of $\{a_{-2\epsilon_{1}}^{\theta'}\mid \theta'\in J\}$ by Lemma \ref{33336}. It follows that $$\varphi=\sum_{\theta\in J}a_{-2\epsilon_{1}}^{\theta}\mathfrak{D}_{\omega_{0}^{\theta}}.$$ Moreover, $\varphi\in \Ider(\mathfrak{g}, Z_{\chi}(\lambda))_{(0)}$.

\textbf{Case 2}. Suppose that there exists $\theta'\in J$ or $\beta'\in \{\pm2\epsilon_{i}, \pm\epsilon_{1}\pm\epsilon_{2}\pm\epsilon_{3}\}$ such that $\beta';\theta';1=p-1$. We first verify that $\varphi_{\bar{0}}\in \Ider(\mathfrak{g}, Z_{\chi}(\lambda))_{(0)}$. There are four possibilities for $\lambda_{1}$ which give rise to $\beta';\theta';1=p-1$, where $\theta'\in J_{1}$.

(1) If $\lambda_{1}=2p+4$, then $2\epsilon_{1};(1,1,1,1);1=p-1$. The computation for the equations
\begin{equation}\label{gongshi2}
\varphi_{\bar{0}}([e_{1}, x_{4}])=0, \varphi_{\bar{0}}([e_{1}, f_{2}])=0, \varphi_{\bar{0}}([e_{1}, f_{3}])=0,
\end{equation}
shows that $a_{2\epsilon_{1}}=0$. Replacing $x$ by $f_{1}$ and $y$ by other basis elements in (\ref{222222}), we obtain that $$\varphi_{\bar{0}}=\sum_{\theta\in J_{1}}a_{-2\epsilon_{1}}^{\theta}\mathfrak{D}_{\omega_{0}^{\theta}}$$ by Lemma \ref{33336}. In other words, $\varphi_{\bar{0}}\in \Ider(\mathfrak{g}, Z_{\chi}(\lambda))_{(0),\bar{0}}$.

(2) If $\lambda_{1}=2p+2,$ then
\begin{equation}\label{gongshi3}
  2\epsilon_{1};\theta_{1};1=\pm2\epsilon_{k};(1,1,1,1);1=\epsilon_{1}\pm\epsilon_{2}\pm\epsilon_{3};\theta_{2};1=p-1,
\end{equation}
where $\theta_{1}\in J_{2}$, $\theta_{2}\in J_{4}$ and $k=2,3$.
Replacing $x$ by $f_{1}$ and $y$ by other basis elements in (\ref{222222}) and Lemma \ref{33336}, we obtain that $$\varphi_{\bar{0}}=\sum_{\theta\in J_{1}\backslash\{(1,1,1,1)\}}a_{-2\epsilon_{1}}^{\theta}\mathfrak{D}_{\omega_{0}^{\theta}}+\frac{1}{\chi(f_{1})^{p}}a_{-2\epsilon_{1}}\mathfrak{D}_{\omega_{0}^{(1,1,1,1)}}.$$
(3) If $\lambda_{1}=2p,$ then
\begin{equation}\label{gongshi6}
    \begin{aligned}
    p-1&=2\epsilon_{1};(0,0,0,0);1=-2\epsilon_{1};(1,1,1,1);1=\pm\epsilon_{k};\theta_{1};1\\
       &=\epsilon_{1}\pm\epsilon_{2}\pm\epsilon_{3};\theta_{2};1=-\epsilon_{1}\pm\epsilon_{2}\pm\epsilon_{3};\theta_{3};1,
    \end{aligned}
\end{equation}
where $k=2,3$, $\theta_{1}\in J_{2}$, $\theta_{2}\in J_{3}\backslash J_{4}$ and $\theta_{3}\in J_{4}$.
Replacing $x$ by $f_{1}$ and $y$ by other basis elements in (\ref{222222}) and Lemma \ref{33336}, we obtain that
$$\varphi_{\bar{0}}=\sum_{\theta\in J_{2}}\frac{1}{\chi(f_{1})^{p}}a_{-2\epsilon_{1}}^{\theta}\mathfrak{D}_{\omega_{0}^{\theta}}
+a_{-2\epsilon_{1}}^{(0,0,0,0)}\mathfrak{D}_{\omega_{0}^{(0,0,0,0)}}+a_{-2\epsilon_{1}}\mathfrak{D}_{\omega_{0}^{(1,1,1,1)}}.$$

(4) If $\lambda_{1}=2p-2$, then
\begin{equation}\label{gongshi5}
     -2\epsilon_{1};\theta_{1};1=\pm2\epsilon_{k};(0,0,0,0);1=-\epsilon_{1}\pm\epsilon_{2}\pm\epsilon_{3};\theta_{2};1=p-1,
\end{equation}
where $\theta_{1}\in J_{2},$ $k=2,3$ and $\theta_{2}\in J_{3}\backslash J_{4}$.
Replacing $x$ by $f_{1}$ and $y$ by other basis elements in (\ref{222222}) and Lemma \ref{33336}, we obtain that
$$\varphi_{\bar{0}}=\sum_{\theta\in J_{1}\backslash \{(0,0,0,0)\}}a_{-2\epsilon_{1}}^{\theta}\mathfrak{D}_{\omega_{0}^{\theta}}
+\frac{1}{\chi(f_{1})^{p}}a_{-2\epsilon_{1}}^{(0,0,0,0)}\mathfrak{D}_{\omega_{0}^{(0,0,0,0)}}.$$
Therefore, $\varphi_{\bar{0}}\in \Ider(\mathfrak{g}, Z_{\chi}(\lambda))_{(0),\bar{0}}$. Then we check that $\varphi_{\bar{1}}\in \Ider(\mathfrak{g}, Z_{\chi}(\lambda))_{(0),\bar{1}}$. There are four possibilities for $\lambda_{1}$ which give rise to $\beta';\theta';1=p-1$, where $\theta'\in J_{3}$.

(1) If $\lambda_{1}=2p+3,$ then $2\epsilon_{1};\theta_{1};1=\epsilon_{1}\pm\epsilon_{2}\pm\epsilon_{3};(1,1,1,1);1=p-1$, where $\theta_{1}\in J_{4}$.
Replacing $x$ by $f_{1}$ and $y$ by $e_{1}$ in (\ref{222222}), we obtain that $\chi(f_{1})^{p}a_{2\epsilon_{1}}^{\theta_{1}}=0$. Since $\chi(f_{1})^{p}\neq0$, we obtain $a_{2\epsilon_{1}}^{\theta_{1}}=0$ for $\theta_{1}\in J_{4}$. This yields $a_{\epsilon_{1}\pm\epsilon_{2}\pm\epsilon_{3}}=0$ by computing the following equations
\begin{equation*}
    \begin{aligned}
\varphi_{\bar{1}}([e_{1},x_{4}])&=0, & \varphi_{\bar{1}}([e_{2}, x_{4}])&=\varphi_{\bar{1}}(x_{2}),\\
  \varphi_{\bar{1}}([e_{3}, x_{4}])&=\varphi_{\bar{1}}(x_{3}), & \varphi_{\bar{1}}([e_{3}, x_{2}])&=\varphi_{\bar{1}}(x_{1}).
    \end{aligned}
\end{equation*}
Replacing $x$ by $f_{1}$ and $y$ by other basis elements in (\ref{222222}), we obtain that $$\varphi_{\bar{1}}=\sum_{\theta\in J_{3}}a_{-2\epsilon_{1}}^{\theta}\mathfrak{D}_{\omega_{0}^{\theta}}$$ by Lemma \ref{33336}. In other words, $\varphi_{\bar{1}}\in \Ider(\mathfrak{g}, Z_{\chi}(\lambda))_{(0),\bar{1}}$.

(2) If $\lambda_{1}=2p+1$, then
\begin{equation}\label{gongshi8}
    \begin{aligned}
    p-1&=2\epsilon_{1};\theta_{1};1=\epsilon_{1}\pm\epsilon_{2}\pm\epsilon_{3};\theta_{3};1\\
    &=\pm\epsilon_{k};\theta_{2};1=-\epsilon_{1}\pm\epsilon_{2}\pm\epsilon_{3};(1,1,1,1);1
    \end{aligned}
\end{equation}
where $\theta_{1}\in J_{3}\backslash J_{4}$, $\theta_{2}\in J_{4}$, $\theta_{3}\in J_{2}$ and $k=2,3$.
Replacing $x$ by $f_{1}$ and $y$ by other basis elements in (\ref{222222}), we obtain that
$$\varphi_{\bar{1}}=\sum_{\theta\in J_{3}\backslash J_{4}}a_{-2\epsilon_{1}}^{\theta}\mathfrak{D}_{\omega_{0}^{\theta}}+\sum_{\theta\in J_{4}}\frac{1}{\chi(f_{1})^{p}}a_{-2\epsilon_{1}}^{\theta}\mathfrak{D}_{\omega_{0}^{\theta}}.$$

(3) If $\lambda_{1}=2p-1$, then
\begin{equation}\label{gongshi7}
    \begin{aligned}
    p-1&=-2\epsilon_{1};\theta_{1};1=\epsilon_{1}\pm\epsilon_{2}\pm\epsilon_{3};(0,0,0,0);1\\
    &=\pm2\epsilon_{k};\theta_{2};1=-\epsilon_{1}\pm\epsilon_{2}\pm\epsilon_{3};\theta_{3};1
    \end{aligned}
\end{equation}
where $\theta_{1}\in J_{4}$, $\theta_{2}\in J_{3}\backslash J_{4}$, $\theta_{3}\in J_{2}$ and $k=2,3$.
Replacing $x$ by $f_{1}$ and $y$ by other basis elements in (\ref{222222}), we obtain that
$$\varphi_{\bar{1}}=\sum_{\theta\in J_{3}\backslash J_{4}}\frac{1}{\chi(f_{1})^{p}}a_{-2\epsilon_{1}}^{\theta}\mathfrak{D}_{\omega_{0}^{\theta}}+\sum_{\theta\in J_{4}}a_{-2\epsilon_{1}}^{\theta}\mathfrak{D}_{\omega_{0}^{\theta}}.$$

(4) If $\lambda_{1}=2p-3$, then $-2\epsilon_{1};\theta;1=-\epsilon_{1}\pm\epsilon_{2}\pm\epsilon_{3};(0,0,0,0);1=p-1$, where $\theta\in J_{3}\backslash J_{4}$. By computing $\varphi_{\bar{1}}([f_{1}, y_{j}])=0$ for $j=1,2,3,4$ and $\chi(f_{1})^{p}\neq0$, we have
\begin{equation}\label{gongshi1}
    \begin{aligned}
    a_{-\epsilon_{1}+\epsilon_{2}+\epsilon_{3}}^{(0,0,0,0)}&=
(1+\alpha)(\lambda_{2}+1)[\frac{\lambda_{3}+1}{2}a_{-2\epsilon_{1}}^{(1,0,0,0)}
+a_{-2\epsilon_{1}}^{(0,1,0,0)}],\\
a_{-\epsilon_{1}+\epsilon_{2}-\epsilon_{3}}^{(0,0,0,0)}&=-(1+\alpha)(\lambda_{2}+1)a_{-2\epsilon_{1}}^{(1,0,0,0)},\\
a_{-\epsilon_{1}-\epsilon_{2}+\epsilon_{3}}^{(0,0,0,0)}&=-(1+\alpha)[(\lambda_{3}+1)a_{-2\epsilon_{1}}^{(1,0,0,0)}
+2a_{-2\epsilon_{1}}^{(0,1,0,0)}],\\
a_{-\epsilon_{1}-\epsilon_{2}-\epsilon_{3}}^{(0,0,0,0)}&=2(1+\alpha)a_{-2\epsilon_{1}}^{(1,0,0,0)}.
    \end{aligned}
\end{equation}
It follows that $\varphi_{\bar{1}}=\sum_{\theta\in J_{3}}a_{-2\epsilon_{1}}^{\theta}\mathfrak{D}_{\omega_{0}^{\theta}}$ by substituting $f_{1}$ for $x$ and other basis elements for $y$ in (\ref{222222}) and $(\ref{gongshi1})$.
In summary we have $\varphi\in\Ider(\mathfrak{g}, Z_{\chi}(\lambda))_{(0)}.$
\end{proof}

For convenience, we define some linear mappings as follows. Let $\psi_{1}$ be an even linear mapping from $\mathfrak{g}$ to $Z_{\chi}(\lambda)$ defined by means of
\begin{align*}
    \psi_{1}(h_{i})&=a_{i}\omega_{0}^{(1,1,1,1)}, i=1,2,3, \\
    \psi_{1}(f_{j})&=a_{-2\epsilon_{j}}\omega_{-2\epsilon_{2}}^{(1,1,1,1)}, j=2,3, \\
    \psi_{1}(e_{1})&=2a_{-2\epsilon_{2}}\omega_{2\epsilon_{1}}^{(1,1,0,0)}-2\alpha a_{-2\epsilon_{3}}\omega_{2\epsilon_{1}}^{(1,0,1,0)}+[(1+\alpha)a_{1}-a_{2}+\alpha a_{3}]\omega_{2\epsilon_{1}}^{(1,0,0,1)}\\
    &-[(1+\alpha)a_{1}-a_{2}-\alpha a_{3}]\omega_{2\epsilon_{1}}^{(0,1,1,0)}-a_{1}\omega_{2\epsilon_{1}}^{(1,1,1,1)},\\
    \psi_{1}(e_{2})&=-a_{2}\omega_{2\epsilon_{2}}^{(1,1,1,1)}, \\
    \psi_{1}(e_{3})&=-a_{3}\omega_{2\epsilon_{3}}^{(1,1,1,1)},\\
    \psi_{1}(x_{1})&=[(1+\alpha)a_{1}-a_{2}-\alpha a_{3}]\omega_{\epsilon_{1}+\epsilon_{2}+\epsilon_{3}}^{(1,1,1,0)}+[(1+\alpha)a_{1}-a_{2}+\alpha a_{3}]\omega_{\epsilon_{1}+\epsilon_{2}+\epsilon_{3}}^{(1,1,0,1)}\\
    &-[(1+\alpha)a_{1}+a_{2}-\alpha a_{3}]\omega_{\epsilon_{1}+\epsilon_{2}+\epsilon_{3}}^{(1,0,1,1)}-[(1+\alpha)a_{1}+a_{2}+\alpha a_{3}]\omega_{\epsilon_{1}+\epsilon_{2}+\epsilon_{3}}^{(0,1,1,1)},\\
    \psi_{1}(x_{2})&=-2\alpha a_{-2\epsilon_{3}}\omega_{\epsilon_{1}+\epsilon_{2}-\epsilon_{3}}^{(1,1,1,0)}+[(1+\alpha)a_{1}-a_{2}+\alpha a_{3}]\omega_{\epsilon_{1}+\epsilon_{2}-\epsilon_{3}}^{(1,1,0,1)}\\
    &+2\alpha a_{-2\epsilon_{3}}\omega_{\epsilon_{1}+\epsilon_{2}-\epsilon_{3}}^{(1,0,1,1)}-[(1+\alpha)a_{1}+a_{2}+\alpha a_{3}]
    \omega_{\epsilon_{1}+\epsilon_{2}-\epsilon_{3}}^{(0,1,1,1)},\\
    \psi_{1}(x_{3})&=-2a_{-2\epsilon_{2}}\omega_{\epsilon_{1}-\epsilon_{2}+\epsilon_{3}}^{(1,1,1,0)}
    -2a_{-2\epsilon_{2}}\omega_{\epsilon_{1}-\epsilon_{2}+\epsilon_{3}}^{(1,1,0,1)}-[(1+\alpha)a_{1}+a_{2}-\alpha a_{3}]\omega_{\epsilon_{1}-\epsilon_{2}+\epsilon_{3}}^{(1,0,1,1)}\\
    &-[(1+\alpha)a_{1}+a_{2}+\alpha a_{3}]\omega_{\epsilon_{1}-\epsilon_{2}+\epsilon_{3}}^{(0,1,1,1)},\\
    \psi_{1}(x_{4})&=-2a_{-2\epsilon_{2}}\omega_{\epsilon_{1}-\epsilon_{2}-\epsilon_{3}}^{(1,1,0,1)}+2\alpha a_{-2\epsilon_{3}}\omega_{\epsilon_{1}-\epsilon_{2}-\epsilon_{3}}^{(1,0,1,1)}-[(1+\alpha)a_{1}+a_{2}+\alpha a_{3}]\omega_{\epsilon_{1}-\epsilon_{2}-\epsilon_{3}}^{(0,1,1,1)}.
    \end{align*}
Let $\psi_{2}$ be an even mapping from $\mathfrak{g}$ to $Z_{\chi}(\lambda)$ defined by means of
\begin{equation*}
    \begin{aligned}
    \psi_{2}(e_{1})&=2\alpha a_{2\epsilon_{3}}\omega_{2\epsilon_{1}}^{(0,1,0,1)},\\
    \psi_{2}(e_{3})&=a_{2\epsilon_{3}}\omega_{2\epsilon_{3}}^{(1,1,1,1)},\\
    \psi_{2}(x_{1})&=-2\alpha a_{2\epsilon_{3}}\omega_{\epsilon_{1}+\epsilon_{2}+\epsilon_{3}}^{(1,1,0,1)}+2\alpha a_{2\epsilon_{3}}\omega_{\epsilon_{1}+\epsilon_{2}+\epsilon_{3}}^{(0,1,1,1)},\\
    \psi_{2}(x_{3})&=2\alpha a_{2\epsilon_{3}}\omega_{\epsilon_{1}-\epsilon_{2}+\epsilon_{3}}^{(0,1,1,1)}.
    \end{aligned}
\end{equation*}
Let $\psi_{3}$ be an even mapping from $\mathfrak{g}$ to $Z_{\chi}(\lambda)$ defined by means of
\begin{equation*}
    \begin{aligned}
    \psi_{3}(e_{1})&=-2a_{2\epsilon_{2}}\omega_{2\epsilon_{2}}^{(0,0,1,1)},\\
    \psi_{3}(e_{2})&=a_{2\epsilon_{2}}\omega_{2\epsilon_{2}}^{(1,1,1,1)},\\
    \psi_{3}(x_{1})&=2a_{2\epsilon_{2}}\omega_{\epsilon_{1}+\epsilon_{2}+\epsilon_{3}}^{(1,0,1,1)}
    +2a_{2\epsilon_{2}}\omega_{\epsilon_{1}+\epsilon_{2}+\epsilon_{3}}^{(0,1,1,1)},\\
    \psi_{3}(x_{2})&=2a_{2\epsilon_{2}}\omega_{\epsilon_{1}+\epsilon_{2}-\epsilon_{3}}^{(0,1,1,1)}.
    \end{aligned}
\end{equation*}
Let $\psi_{4}$ be an odd mapping from $\mathfrak{g}$ to $Z_{\chi}(\lambda)$ defined by means of
\begin{equation*}
    \begin{aligned}
    \psi_{4}(e_{1})&=a_{2\epsilon_{1}}^{(1,1,1,0)}\omega_{2\epsilon_{1}}^{(1,1,1,0)},& \psi_{4}(x_{1})&=-\frac{\lambda_{2}+1}{2}\frac{\lambda_{3}+1}{2}
    a_{2\epsilon_{1}}^{(1,1,1,0)}\omega_{\epsilon_{1}+\epsilon_{2}+\epsilon_{3}}^{(1,1,1,1)},\\
    \psi_{4}(x_{2})&=\frac{\lambda_{2}+1}{2}a_{2\epsilon_{1}}^{(1,1,1,0)}\omega_{\epsilon_{1}+\epsilon_{2}-\epsilon_{3}}^{(1,1,1,1)}, &
    \psi_{4}(x_{3})&=\frac{\lambda_{3}+1}{2}a_{2\epsilon_{1}}^{(1,1,1,0)}\omega_{\epsilon_{1}-\epsilon_{2}+\epsilon_{3}}^{(1,1,1,1)},\\
    \psi_{4}(x_{4})&=-a_{2\epsilon_{1}}^{(1,1,1,0)}\omega_{\epsilon_{1}-\epsilon_{2}-\epsilon_{3}}^{(1,1,1,1)}.&&
    \end{aligned}
\end{equation*}
We can easily check that $\psi_{1}, \psi_{2}, \psi_{3}$ and $\psi_{4}$ are derivations from $\mathfrak{g}$ to $Z_{\chi}(\lambda)$.

\begin{proposition}\label{3.5}
Suppose that $\varphi\in \Der(\mathfrak{g}, Z_{\chi}(\lambda))_{(0)}$. If $\chi(f_{1})^{p}=0$, then there exist $\varphi_{1}$ and $\varphi_{2}$ such that $\varphi=\varphi_{1}+\varphi_{2},$ where $\varphi_{1}\in \Ider(\mathfrak{g}, Z_{\chi}(\lambda))_{(0)}$ and $\varphi_{2}\in \mathrm{span}_{\mathbb{F}}\{\psi_{1},\psi_{2},\psi_{3},\psi_{4}\}$.
\end{proposition}

\begin{proof}
Suppose that $\varphi=\varphi_{\bar{0}}+\varphi_{\bar{1}}$, where $\varphi_{\bar{i}}\in \Der(\mathfrak{g}, Z_{\chi}(\lambda))_{(0),\bar{i}}$ for $i=0,1$. By Proposition \ref{33335}, we need only to describe $\varphi$ in case there exists $\theta'\in J$ or $\beta'\in \{\pm2\epsilon_{i}, \pm\epsilon_{1}\pm\epsilon_{2}\pm\epsilon_{3}\}$ such that $\beta';\theta';1=p-1$.

(a) We first consider $\varphi_{\bar{0}}\in \Der(\mathfrak{g}, Z_{\chi}(\lambda))_{(0),\bar{0}}$.
By Lemma \ref{222}, we shall discuss the cases $\lambda_{1}=2p, 2p\pm2, 2p+4$, respectively.

(a.1)  If $\lambda_{1}=2p+4,$ then $2\epsilon_{1};(1,1,1,1);1=p-1$. By virtue of (\ref{gongshi2}), we note that $a_{2\epsilon_{1}}=0$. Replacing $x$ by $f_{1}$ and $y$ by other basis elements in (\ref{222222}), we obtain that $\varphi_{\bar{0}}=\sum_{\theta\in J_{1}}a_{-2\epsilon_{1}}^{\theta}\mathfrak{D}_{\omega_{0}^{\theta}}$ by Lemma \ref{33336}. It follows that $\varphi_{\bar{0}}\in \Ider(\mathfrak{g}, Z_{\chi}(\lambda))_{(0),\bar{0}}$.

(a.2) If $\lambda_{1}=2p+2$, then (\ref{gongshi3}) holds. By substituting all basis elements for $x$ and $y$ in $(\ref{222222})$ and Lemma \ref{3.4}, we have $$\varphi_{\bar{0}}=\sum_{\theta\in J_{1}\backslash\{(1,1,1,1)\}}a_{-2\epsilon}^{\theta}\mathfrak{D}_{\omega_{0}^{\theta}}+\psi$$
and $\psi\in \Der(\mathfrak{g}, Z_{\chi}(\lambda))_{(0),\bar{0}}$ is defined by means of
\begin{align*}\label{gongshi4}
    \psi(h_{i})&=\left\{
               \begin{array}{ll}
                 a_{i}w_{0}^{(1,1,1,1)} & \hbox{if $\lambda=(2, -2, -2)$},\\
                 0 & \hbox{otherwise},\\
        \end{array}
       \right.\\
    \psi(f_{1})&=a_{-2\epsilon_{1}}\omega_{-2\epsilon_{1}}^{(1,1,1,1)}, \psi(f_{2})=a_{-2\epsilon_{2}}\omega_{-2\epsilon_{2}}^{(1,1,1,1)},
    \psi(f_{3})=a_{-2\epsilon_{3}}\omega_{-2\epsilon_{3}}^{(1,1,1,1)}, \\
    \psi(e_{1})&=-8(1+\alpha)a_{-2\epsilon_{1}}\omega_{2\epsilon_{1}}^{(0,0,0,0)}+2a_{-2\epsilon_{2}}\omega_{2\epsilon_{1}}^{(1,1,0,0)}
    -2\alpha a_{-2\epsilon_{3}}\omega_{2\epsilon_{1}}^{(1,0,1,0)}\\
    &+[(1+\alpha)a_{1}-a_{2}+\alpha a_{3}-(\lambda_{2}+2)a_{-2\epsilon_{2}}+\alpha(\lambda_{3}+2)a_{-2\epsilon_{3}}]
    \omega_{2\epsilon_{1}}^{(1,0,0,1)}\\
    &+[-(1+\alpha)a_{1}+a_{2}+\alpha a_{3}+(\lambda_{2}+2)a_{-2\epsilon_{2}}+\alpha(\lambda_{3}+2)a_{-2\epsilon_{3}}]
    \omega_{2\epsilon_{1}}^{(0,1,1,0)}\\
    &+[2\alpha a_{2\epsilon_{2}}-\alpha(\lambda_{3}+2)\frac{\lambda_{3}}{2}a_{-2\epsilon_{3}}]\omega_{2\epsilon_{1}}^{(0,1, 0,1)}+[(\lambda_{2}+2)\frac{\lambda_{2}}{2}a_{-2\epsilon_{2}}-2a_{2\epsilon_{3}}]\omega_{2\epsilon_{1}}^{(0,0,1,1)}\\
    &
    +a_{1}\omega_{2\epsilon_{1}}^{(1,1,1,1)}\\
    \psi(e_{2})&=a_{2\epsilon_{2}}\omega_{2\epsilon_{2}}^{(1,1,1,1)},\psi(e_{3})=a_{2\epsilon_{3}}\omega_{2\epsilon_{3}}^{(1,1,1,1)},\\
    \psi(y_{1})&=(1+\alpha)\lambda_{2}(a_{-2\epsilon_{1}}\omega_{-\epsilon_{1}+\epsilon_{2}+\epsilon_{3}}^{(1,0,1,1)}
    -\frac{\lambda_{3}}{2}\omega_{-\epsilon_{1}+\epsilon_{2}+\epsilon_{3}}^{(0,1,1,1)}),\\
    \psi(y_{2})&=(1+\alpha)\lambda_{2}a_{-2\epsilon_{1}}\omega_{-\epsilon_{1}+\epsilon_{2}-\epsilon_{3}}^{(0,1,1,1)},\\
    \psi(y_{3})&=(1+\alpha)(-2a_{-2\epsilon_{1}}\omega_{-\epsilon_{1}-\epsilon_{2}+\epsilon_{3}}^{(1,0,1,1)}
    +\lambda_{3}a_{-2\epsilon_{1}}\omega_{-\epsilon_{1}-\epsilon_{2}+\epsilon_{3}}^{(0,1,1,1)}),\\
    \psi(y_{4})&=-2(1+\alpha)a_{-2\epsilon_{1}}\omega_{-\epsilon_{1}-\epsilon_{2}-\epsilon_{3}}^{(0,1,1,1)},\\
    \psi(x_{1})&=(1+\alpha)[4a_{-2\epsilon_{1}}\omega_{\epsilon_{1}+\epsilon_{2}+\epsilon_{3}}^{(1,0,0,0)}
    -2\lambda_{3}a_{-2\epsilon_{1}}\omega_{\epsilon_{1}+\epsilon_{2}+\epsilon_{3}}^{(0,1,0,0)}
    -2a_{-2\epsilon_{1}}\omega_{\epsilon_{1}+\epsilon_{2}+\epsilon_{3}}^{(0,0,1,0)}\\
    &+\lambda_{2}\lambda_{3}a_{-2\epsilon_{1}}\omega_{\epsilon_{1}+\epsilon_{2}+\epsilon_{3}}^{(0,0,0,1)}]+[(1+\alpha)a_{1}-a_{2}-\alpha a_{3}]\omega_{\epsilon_{1}+\epsilon_{2}+\epsilon_{3}}^{(1,1,1,0)}\\
    &+[(1+\alpha)a_{1}-a_{2}+\alpha a_{3}-2\alpha a_{2\epsilon_{3}}]\omega_{\epsilon_{1}+\epsilon_{2}+\epsilon_{3}}^{(1,1,0,1)}\\
    &-[(1+\alpha)a_{1}+a_{2}-\alpha a_{3}-2a_{2\epsilon_{2}}]\omega_{\epsilon_{1}+\epsilon_{2}+\epsilon_{3}}^{(1,0,1,1)}\\
    &-[(1+\alpha)a_{1}+a_{2}+\alpha a_{3}+\lambda_{3}a_{2\epsilon_{2}}+\alpha\lambda_{2}a_{2_{\epsilon_{3}}}]\omega_{\epsilon_{1}+\epsilon_{2}+\epsilon_{3}}^{(0,1,1,1)},\\
    \psi(x_{2})&=4(1+\alpha)a_{-2\epsilon_{1}}\omega_{\epsilon_{1}+\epsilon_{2}-\epsilon_{3}}^{(0,1,0,0)}
    -2(1+\alpha)\lambda_{2}a_{-2\epsilon_{1}}\omega_{\epsilon_{1}+\epsilon_{2}-\epsilon_{3}}^{(0,0,0,1)}\\
    &-2\alpha a_{-2\epsilon_{3}}\omega_{\epsilon_{1}+\epsilon_{2}-\epsilon_{3}}^{(1,1,1,0)}
    -\alpha\lambda_{2}a_{-2\epsilon_{3}}\omega_{\epsilon_{1}+\epsilon_{2}-\epsilon_{3}}^{(1,0,1,1)}\\
    &+[(1+\alpha)a_{1}-a_{2}+\alpha a_{3}+\alpha(\lambda_{3}+2)a_{-2\epsilon_{3}}]\omega_{\epsilon_{1}+\epsilon_{2}-\epsilon_{3}}^{(1,1,0,1)}\\
    &-[(1+\alpha)a_{1}+a_{2}+\alpha a_{3}-2a_{-2\epsilon_{2}}-\alpha\lambda_{2}\frac{\lambda_{3}+2}{2}a_{-2\epsilon_{3}}]
    \omega_{\epsilon_{1}+\epsilon_{2}-\epsilon_{3}}^{(0,1,1,1)},\\
    \psi(x_{3})&=4(1+\alpha)a_{-2\epsilon_{1}}\omega_{\epsilon_{1}-\epsilon_{2}+\epsilon_{3}}^{(0,0,0,1)}
    -2a_{-2\epsilon_{2}}\omega_{\epsilon_{1}-\epsilon_{2}+\epsilon_{3}}^{(1,1,1,0)}
    +\lambda_{3}a_{-2\epsilon_{2}}\omega_{\epsilon_{1}-\epsilon_{2}+\epsilon_{3}}^{(1,1,0,1)}\\
    &-[(1+\alpha)a_{1}+a_{2}-\alpha a_{3}+(\lambda_{2}+2)a_{-2\epsilon_{2}}]\omega_{\epsilon_{1}-\epsilon_{2}+\epsilon_{3}}^{(1,0,1,1)}\\
    &-[(1+\alpha)a_{1}+a_{2}+\alpha a_{3}-(\lambda_{2}+2)\frac{\lambda_{3}}{2}a_{-2\epsilon_{2}}-2\alpha a_{2\epsilon_{3}}]\omega_{\epsilon_{1}-\epsilon_{2}+\epsilon_{3}}^{(0,1,1,1)},\\
    \psi(x_{4})&=4(1+\alpha)a_{-2\epsilon_{1}}\omega_{\epsilon_{1}-\epsilon_{2}-\epsilon_{3}}^{(0,0,0,1)}
    -2a_{-2\epsilon_{2}}\omega_{\epsilon_{1}-\epsilon_{2}-\epsilon_{3}}^{(1,1,0,1)}+2\alpha a_{-2\epsilon_{3}}\omega_{\epsilon_{1}-\epsilon_{2}-\epsilon_{3}}^{(1,0,1,1)}\\
    &-[(1+\alpha)a_{1}+a_{2}+\alpha a_{3}+(\lambda_{2}+2)a_{-2\epsilon_{2}}+\alpha(\lambda_{3}+2)a_{-2\epsilon_{3}}]\omega_{\epsilon_{1}-\epsilon_{2}-\epsilon_{3}}^{(0,1,1,1)}.
    \end{align*}

(a.2.1) For the case $\lambda=(2p+2, 2p-2, 2p-2)$ and $\chi(f_{2})^{p}=\chi(f_{3})^{p}=0$, the inner derivation $\mathfrak{D}_{\omega_{0}}^{(1,1,1,1)}=0$. It follows that $\psi$ is an outer superderivation. Moreover, $\psi=\psi_{1}$.

(a.2.2) For the case $\lambda=(2p+2, 2p-2, 2p)$ and $\chi(f_{2})^{p}=\chi(f_{3})^{p}=0$, then $a_{1}=a_{2}=a_{3}=a_{-2\epsilon_{1}}=a_{-2\epsilon_{2}}=0$.
It follows that $\psi$ decomposes $$\psi=a_{-2\epsilon_{3}}\mathfrak{D}_{\omega_{0}}^{(1,1,1,1)}+\psi_{2}.$$ If $\psi_{2}$ is an inner derivation, then assume that $\psi_{2}=\sum_{\theta}b^{\theta}\mathfrak{D}_{\omega_{0}^{\theta}},$ where $b^{\theta}\in \mathbb{F}$.
Since $\psi_{2}(f_{1})=0$. It follows that $b^{\theta}=0$ for $\theta\in J_{1}\backslash\{(1,1,1,1)\}$.
Then $\psi_{2}=b^{(1,1,1,1)}\mathfrak{D}_{\omega_{0}}^{(1,1,1,1)}$. By virtue of $\lambda_{3}=2p$, we have $\mathfrak{D}_{\omega_{0}}^{(1,1,1,1)}=0$, contradicting the definition of $\psi_{2}$. Hence $\psi_{2}$ is an outer superderivation.

(a.2.3) For the case $\lambda=(2p+2, 2p, 2p-2)$ and $\chi(f_{2})^{p}=\chi(f_{3})^{p}=0$, then $a_{1}=a_{2}=a_{3}=a_{-2\epsilon_{1}}=a_{-2\epsilon_{3}}=0$. It follows that $\psi$ decomposes $$\psi=a_{-2\epsilon_{2}}\mathfrak{D}_{\omega_{0}}^{(1,1,1,1)}+\psi_{3}.$$ We have $\psi_{3}$ is an outer superderivation, the proof is similar to the proof of $\psi_{2}$.

(a.2.4) For the other cases of $\lambda_{2}$ and $\lambda_{3}$, we could check that $\psi=a_{-2\epsilon_{1}}\mathfrak{D}_{\omega_{0}}^{(1,1,1,1)}$ by a similar calculation as above, that is, $\psi$ is an inner derivation.

(a.3) If $\lambda_{1}=2p$, then (\ref{gongshi6}) holds. By substituting $x$ and $y$ for all basis elements in (\ref{222}), we obtain that
\begin{equation*}
    \begin{aligned}
    \varphi&=a_{-2\epsilon_{1}}^{(0,0,0,0)}\mathfrak{D}_{\omega_{0}^{(0,0,0,0)}}+a_{-\epsilon_{1}-\epsilon_{2}-\epsilon_{3}}^{(1,1,0,1)}\mathfrak{D}_{\omega_{0}^{(1,1,0,0)}}
    +a_{-\epsilon_{1}-\epsilon_{2}-\epsilon_{3}}^{(1,0,1,1)}\mathfrak{D}_{\omega_{0}^{(1,0,1,0)}}\\
    &-[\frac{(\lambda_{2}+2)}{2}a_{-\epsilon_{1}-\epsilon_{2}-\epsilon_{3}}^{(1,1,0,1)}+a_{-\epsilon_{1}+\epsilon_{2}-\epsilon_{3}}^{(1,1,0,1)}]
    \mathfrak{D}_{\omega_{0}^{(1,0,0,1)}}+a_{-\epsilon_{1}+\epsilon_{2}+\epsilon_{3}}^{(1,1,1,0)}D_{\omega_{0}^{(0,1,1,0)}}\\
    &+(a_{-\epsilon_{1}+\epsilon_{2}+\epsilon_{3}}^{(1,1,0,1)}+\frac{\lambda_{3}}{2}a_{-\epsilon_{1}+\epsilon_{2}-\epsilon_{3}}^{(1,1,0,1)})
    \mathfrak{D}_{\omega_{0}^{(0,1,0,1)}}+a_{-2\epsilon_{1}}\mathfrak{D}_{\omega_{0}^{(1,1,1,1)}}\\
    &+(a_{-\epsilon_{1}+\epsilon_{2}+\epsilon_{3}}^{(1,0,1,1)}+\frac{\lambda_{2}}{2}a_{-\epsilon_{1}-\epsilon_{2}+\epsilon_{3}}^{(1,0,1,1)})
    \mathfrak{D}_{\omega_{0}^{(0,0,1,1)}}.
    \end{aligned}
\end{equation*}
It follows that $\varphi\in \Ider(\mathfrak{g}, Z_{\chi}(\lambda))_{(0),\bar{0}}$.

(a.4) If $\lambda_{1}=2p-2$, then (\ref{gongshi5}) holds. By substituting $x$ and $y$ for all basis elements in (\ref{222}), we obtain that
\begin{equation*}
    \begin{aligned}
    \varphi&=\sum_{\theta\in J_{1}\backslash \{(1,1,0,0)\}}a_{-2\epsilon_{1}}^{\theta}\mathfrak{D}_{\omega_{0}^{\theta}}\\
    &+[a_{-\epsilon_{1}+\epsilon_{2}+\epsilon_{3}}^{(1,0,0,0)}-(1+\alpha)(\lambda_{2}+2)a_{-2\epsilon_{1}}^{(1,1,0,0)}]\mathfrak{D}_{\omega_{0}^{(1,1,0,0)}}.
    \end{aligned}
\end{equation*}
It shows that $\varphi\in \Ider(\mathfrak{g}, Z_{\chi}(\lambda))_{(0),\bar{0}}$.

(b) We consider $\varphi_{\bar{1}}\in \Der(\mathfrak{g}, Z_{\chi}(\lambda))_{(0),\bar{1}}$. By Lemma \ref{222}, we shall discuss the cases $\lambda_{1}=2p\pm1, 2p\pm3$, respectively.

(b.1) If $\lambda_{1}=2p-3$, then $-2\epsilon_{1};\theta;1=-\epsilon_{1}\pm\epsilon_{2}\pm\epsilon_{3};(0,0,0,0);1=p-1$, where $\theta\in J_{3}\backslash J_{4}$. By substituting $x$ and $y$ for all basis elements in (\ref{222}), then $(\ref{gongshi1})$ holds. It follows that $$\varphi_{\bar{1}}=\sum_{\theta\in J_{3}}a_{-2\epsilon_{1}}^{\theta}\mathfrak{D}_{\omega_{0}}^{\theta}.$$

(b.2) If $\lambda_{1}=2p-1$, then (\ref{gongshi7}) holds. By substituting $x$ and $y$ for all basis elements in (\ref{222}), we obtain that
\begin{equation*}
    \begin{aligned}
    \varphi_{\bar{1}}&=\sum_{\theta\in J_{4}}a_{-2\epsilon_{1}}^{\theta}\mathfrak{D}_{\omega_{0}^{\theta}}+a_{-\epsilon_{1}-\epsilon_{2}-\epsilon_{3}}^{(1,0,0,1)}\mathfrak{D}_{\omega_{0}^{(1,0,0,0)}}\\
    &+[a_{-\epsilon_{1}-\epsilon_{2}-\epsilon_{3}}^{(0,1,0,1)}-2(1+\alpha)a_{-2\epsilon_{1}}^{(1,1,0,1)}]\mathfrak{D}_{\omega_{0}^{(0,1,0,0)}}\\
    &+[a_{-\epsilon_{1}-\epsilon_{2}-\epsilon_{3}}^{(0,0,1,1)}-2(1+\alpha)a_{-2\epsilon_{1}}^{(1,0,1,1)}]\mathfrak{D}_{\omega_{0}^{(0,0,1,0)}}\\
    &-(\frac{\lambda_{2}+1}{2}a_{-\epsilon_{1}-\epsilon_{2}-\epsilon_{3}}^{(0,1,0,1)}+a_{-\epsilon_{1}+\epsilon_{2}-\epsilon_{3}}^{(0,1,0,1)})\mathfrak{D}_{\omega_{0}^{(0,0,0,1)}}.
    \end{aligned}
\end{equation*}
It follows that $\varphi_{\bar{1}}$ is an inner derivation.

(b.3) If $\lambda_{1}=2p+1$, then (\ref{gongshi8}) holds. By substituting $x$ and $y$ for all basis elements in (\ref{222}), we obtain that
\begin{equation*}
    \begin{aligned}
    \varphi_{\bar{1}}&=\sum_{\theta\in J_{3}\backslash J_{4}}a_{-2\epsilon_{1}}^{\theta}\mathfrak{D}_{\omega_{0}^{\theta}}+a_{-\epsilon_{1}-\epsilon_{2}-\epsilon_{3}}\mathfrak{D}_{\omega_{0}^{(1,1,1,0)}}\\
    &-[a_{-\epsilon_{1}-\epsilon_{2}+\epsilon_{3}}+\frac{\lambda_{3}+1}{2}a_{-\epsilon_{1}-\epsilon_{2}-\epsilon_{3}}]\mathfrak{D}_{\omega_{0}^{(1,1,0,1)}}\\
    &+[a_{-\epsilon_{1}+\epsilon_{2}-\epsilon_{3}}+\frac{\lambda_{2}+1}{2}a_{-\epsilon_{1}-\epsilon_{2}-\epsilon_{3}}]\mathfrak{D}_{\omega_{0}^{(1,0,1,1)}}\\
    &-[\frac{\lambda_{2}+1}{2}\frac{\lambda_{3}+1}{2}a_{-\epsilon_{1}-\epsilon_{2}-\epsilon_{3}}+\frac{\lambda_{2}+1}{2}
    a_{-\epsilon_{1}-\epsilon_{2}+\epsilon_{3}}+a_{-\epsilon_{1}+\epsilon_{2}+\epsilon_{3}}\\
    &\quad+\frac{\lambda_{3}+1}{2}a_{-\epsilon_{1}+\epsilon_{2}-\epsilon_{3}}]\mathfrak{D}_{\omega_{0}^{(0,1,1,1)}}.
    \end{aligned}
\end{equation*}
It follows that $\varphi_{\bar{1}}$ is an inner derivation.

(b.4) If $\lambda_{1}=2p+3$, then $2\epsilon_{1};\theta_{1};1=\epsilon_{1}\pm\epsilon_{2}\pm\epsilon_{3};(1,1,1,1);1=p-1$, where $\theta_{1}\in J_{4}$. By substituting $x$ and $y$ for all basis elements in (\ref{222}),
we obtain that $$\varphi_{\bar{1}}=\sum_{\theta\in J_{3}}a_{-2\epsilon_{1}}\mathfrak{D}_{\omega_{0}}^{\theta}+\psi.$$
If $\lambda_{2}=\lambda_{3}=2p-3$ and $\chi(f_{2})^{p}=\chi(f_{3})^{p}=0$, then $\psi=\psi_{4}$. Suppose that $\psi_{4}$ is an inner derivation. We may assume that $\psi_{4}=\sum_{\theta\in J_{3}}b^{\theta}\mathfrak{D}_{\omega_{0}}^{\theta}$. Since $\psi_{4}(f_{1})=0$.
It follows that $b^{\theta}=0$ for $\theta\in J_{3}$, contradicting the definition of $\psi_{4}$. Hence $\psi_{4}$ is an outer superderivation. For other case of $\lambda_{2}, \lambda_{3}$ and $\chi$, we have $\psi_{4}=0$. It follows that $\varphi_{\bar{1}}=\sum_{\theta\in J_{3}}a_{-2\epsilon_{1}}\mathfrak{D}_{\omega_{0}}^{\theta}$. Moreover, $\varphi_{\bar{1}}\in \Ider(\mathfrak{g},Z_{\chi}(\lambda))_{(0),\bar{1}}$.
In summary we obtain the desired result.

\end{proof}

\begin{theorem}\label{3.6}
Let $\mathfrak{g}=D(2, 1; \alpha).$ For $i=1,2,3$ we have $$\H^{1}(\mathfrak{g}, Z_{\chi}(\lambda))=\left\{
               \begin{array}{ll}
                 \mathbb{F}\psi_{1} & \hbox{if $\lambda=(2p+2, 2p-2, 2p-2)$ and $\chi(f_{i})^{p}=0$},\\
                 \mathbb{F}\psi_{2} & \hbox{if $\lambda=(2p+2, 2p-2, 2p)$ and $\chi(f_{i})^{p}=0$},\\
                 \mathbb{F}\psi_{3} & \hbox{if $\lambda=(2p+2, 2p, 2p-2)$ and $\chi(f_{i})^{p}=0$},\\
                 \mathbb{F}\psi_{4} & \hbox{if $\lambda=(2p+3, 2p-3, 2p-3)$ and $\chi(f_{i})^{p}=0$},\\
                 0 & \hbox{otherwise}.
                 \end{array}
                 \right.$$
Moreover, $$\mathrm{sdim}\H^{1}(\mathfrak{g}, Z_{\chi}(\lambda))=\left\{
               \begin{array}{ll}
                 (6,0) & \hbox{if $\lambda=(2p+2, 2p-2, 2p-2)$ and $\chi(f_{i})^{p}=0$},\\
                 (1,0) & \hbox{if $\lambda=(2p+2, 2p-2, 2p)$ and $\chi(f_{i})^{p}=0$},\\
                 (1,0) & \hbox{if $\lambda=(2p+2, 2p, 2p-2)$ and $\chi(f_{i})^{p}=0$},\\
                 (0,1) & \hbox{if $\lambda=(2p+3, 2p-3, 2p-3)$ and $\chi(f_{i})^{p}=0$},\\
                 0 & \hbox{otherwise}.
                 \end{array}
                 \right.$$
\end{theorem}
\begin{proof}
By the definition of $\psi_{1}$, we may obtain that $\psi_{1}$ is a linear mapping with six variables $a_{i}$ and $a_{-2\epsilon_{1}},$ $i=1,2,3$. Then this follows directly from the proof of Proposition \ref{33335} and Proposition \ref{3.5}.
\end{proof}

\subsection*{Acknowledgements}

The second named author was supported by the NSF of China (12061029), the
NSF of Hainan Province (120RC587) and the
NSF of Heilongjiang Province (YQ2020A005).

\end{document}